\documentclass[12 pt]{amsart}
\usepackage{amsmath}
\usepackage{amsfonts}
\newtheorem{lemma}{Lemma}
\newtheorem{theorem}{Theorem}

\newtheorem{corollary}{Corollary}

\newtheorem{defn}{Definition}
\numberwithin{equation}{section}
\newcommand{\Ordo}[1]{{O(#1)}}
\newcommand{\ordo}[1]{{o(#1)}}
\newcommand{\R}{\mathbb{R}}

\newcommand{\N}{\mathbb{N}}

\newcommand{\inr}{\mbox{int}}
\newcommand{\spa}{\mbox{span}}

\newcommand{\sppan}{\mbox{span}}

\renewcommand{\dim}{\mbox{dim}_H}
\title[local dimension for non-uniformly expanding maps]{Multifractal analysis of
weak Gibbs measures for non-uniformly expanding $C^1$ maps}
\author{Thomas Jordan}
\address{Thomas Jordan\\Department of Mathematics\\ The University of Bristol\\
University Walk\\Clifton\\ Bristol\\BS8 1TW\\UK}
\email{thomas.jordan@bristol.ac.uk}
\author{Micha\l\ Rams}
\address{Micha\l\ Rams\\Institute of Mathematics\\ Polish Academy of Sciences\\ ul.
\'Sniadeckich 8, 00-956 Warszawa\\ Poland }
\email{M.Rams@impan.gov.pl}
\thanks{We  would like to thank the referee of this paper for their detailed comments which certainly helped improve the paper. The research of M.R. was supported by grants EU FP6 ToK SPADE2, EU FP6 RTN
CODY and MNiSW grant 'Chaos, fraktale i dynamika konforemna'. The
research was started during a visit of M. R. to Bristol. M. R.
would like to thank Bristol University for the hospitality shown
during his visit.}
\begin{document}
\begin{abstract}
We will consider the local dimension spectrum of a weak Gibbs
measure on a $C^1$ non-uniformly hyperbolic system of
Manneville-Pomeau type. We will present the spectrum in three
ways: using invariant measures, ergodic invariant measures
supported on hyperbolic sets and equilibrium states. We are also
proving analyticity of the spectrum under additional assumptions.
All three presentations are well known for smooth uniformly
hyperbolic systems.
\end{abstract}
\maketitle The theory of multifractal analysis for hyperbolic
conformal dynamical systems is now extremely well developed. There
are complete results for local dimension of Gibbs measures,
Lyapunov spectra and Birkhoff spectra. For early results on the
local dimension spectra see \cite{CM} and \cite{R}, for a general
description see \cite{P} and for more specific and very general
results see \cite{BS} and \cite{O}. However, the picture is not
complete for non-uniformly hyperbolic systems. There are results
for specific non-uniformly expanding systems by Kesseb\"{o}hmer
and Stratmann (\cite{KS1} and \cite{KS2}) and in the case of
complex dynamics by Byrne, \cite{B}.

As our interest is in the local dimension spectrum for invariant
measures, five papers are of special interest for us. The first
results about local dimension were in the papers by Nakaishi and
by Pollicott and Weiss, \cite{N},\cite{PW}. They obtained the
local dimension spectrum for the measure of maximal entropy for a
restricted class of parabolic maps without critical points.
Stratmann and Urba\'{n}ski in \cite{SU2} and \cite{SU} considered
complex (hence analytic) maps with possible critical points.
However, the analyticity assumption in these papers could be
probably reduced to assumptions from section 3 of \cite{HMU}. In Theorem 4 of
\cite{Yuri} Yuri computes a portion of the multifractal spectra for weak Gibbs measures of some non-uniformly hyperbolic systems.

The aim of this paper is to obtain a complete spectrum for the
local dimension of weak Gibbs measures for $C^1$ non-uniformly
hyperbolic systems. We will include results about the points where
the local dimension is infinite, a phenomenon which does not occur
in the uniformly hyperbolic setting.   We will be considering
systems with parabolic periodic points but no critical points.
Well known examples of such maps include the Manneville-Pomeau map
and the Farey map.

The methods we use are adapted from the papers \cite{GR} and
\cite{JJOP} where the Lyapunov and Birkhoff spectra of such maps
are considered. For most of the paper we work directly with the
original system, without inducing, which lets us omit the usual
assumptions about behaviour of the map around parabolic points
(except for  Theorem \ref{thm:anal} which deals with analyticity
of the spectra).

\section{Notation and results}
We consider non-uniformly expanding one-dimensional Markov maps.
More precisely, let $I=[0,1]$. Let $\{I_i\}, i=1,\ldots,p$ be
closed subintervals of $I$ with disjoint interiors. Let $A$ be a
$p\times p$ matrix consisting of $0$'s and $1$'s where there
exists $k\in\N$ such that for all $i,j$ $A^{k+1}(i,j)>0$. Let
$T_i: I\to \R$ be $C^1$ bijective diffeomorphisms with closed
domains $J_i\subset I_i$ for which $I_j\subset T_i(J_i)$ if
$A(i,j)=1$ and $T_i(J_i)\cap \inr I_j=\emptyset$ if $A(i,j)=0$. We
will let $\Lambda_0=\cup_i J_i$ and let $T:\Lambda_0\rightarrow I$
be defined as $T_i$ on $J_i$. When $J_i\cap J_j=\{x\}$ and $i<j$
we will take $T(x)=T_i(x)$. We will assume that $|T_i'(x)|\geq 1$
at every point $x\in J_i$ and that there are at most countably
many points with derivative $\pm 1$. We will denote by $\Lambda$
the set of points whose trajectory never leaves $\Lambda_0$.

We will allow the existence of parabolic periodic orbits\\
$\{x,T(x),\ldots,T^{m-1}(x),T^m(x)=x\}$ where the derivative of
$T$ will be $\pm 1$ at all points on the orbit. The existence of a
parabolic orbit implies that the map is non-uniformly expanding.

Let $\Sigma=\{1,\ldots,p\}^\N$ be the full shiftspace with the
usual left shift $\sigma$ and $\Sigma_{A}\subset\Sigma$ the
subshift with respect to the matrix $A$. For
$\underline{i}\in\Sigma$ we will denote by $i_n$ its $n$-th
element and by $i^n$ the sequence of its first $n$ elements. For
some $0<\beta<1$ we will use the metric $d_{\kappa}$ on $\Sigma_A$
given by
$d_{\kappa}(\underline{i},\underline{j})=\kappa^{|\underline{i}\wedge\underline{j}|}$
where $|\underline{i}\wedge\underline{j}|=\inf\{m\in\N:i_m\neq
j_m\}-1$. By our assumptions, $(\Sigma_A,\sigma)$ is topologically
transitive. Let $\Pi:\Sigma_A\rightarrow \Lambda$ be defined by
$$\Pi(\underline{i})=\lim_{n\rightarrow\infty}T_{i_1}^{-1}\circ\cdots
\circ T_{i_n}^{-1}(\Lambda)$$
(the limit is always one point, see the proof of Lemma 2.1 in \cite{U}). The local
dimension of a measure $\mu$ at a point $x$ is defined by
$$d_{\mu}(x)=\lim_{r\rightarrow 0}\frac{\log\mu(B(x,r))}{\log r}$$
when this limit exists.

Let $\phi: \Sigma_A\to \R$ be a continuous potential. We call a $\sigma$-invariant probabilistic measure
$\tilde{\nu}$ supported on $\Sigma_A$ {\em weak Gibbs}
for the potential $\phi$ if there exists a constant $P$ and a
decreasing sequence $\{k_n\}$ such that $\lim_{n\to\infty} k_n=0$
and for every $\underline{i}\in \Sigma_A$ and for every $n$
\[
\exp(-nk_n) \leq \frac {\tilde{\nu}([i_1\ldots i_n])} {\exp(S_n
\phi(\underline{i}) -nP)} \leq \exp(nk_n).
\]
This is similar to the definition in \cite{Yuri00} except we require the result to hold for all sequences not just for $\tilde{\nu}$-almost all sequences.
The existence of such weak Gibbs measures for all continuous potentials is
established in \cite{Kess01}.

For the rest of the paper let $\tilde{\nu}$ be the weak Gibbs
measure for the potential $\phi$ and
$\nu=\tilde{\nu}\circ\Pi^{-1}$.  Our goal
will be to describe the local dimension spectrum function
$\alpha\to\dim X_\alpha$, where
\[
X_{\alpha}=\{x\in \Lambda:d_{\nu}(x)=\alpha\}.
\]
Since $\tilde{\nu}$ is invariant  under the map $\sigma$ it
follows that $\nu$ is invariant  under $T$. We will assume that
$P(\phi)=0$ and $\phi(x)<0$ at each $x\in\Sigma_A$. Note here that
this assumption is not very restrictive in the narrower class
of H\"{o}lder potentials. If $\phi$ is H\"{o}lder continuous and
does not satisfy this assumption then it is possible to add a
constant and a coboundary to obtain a new potential which will
satisfy these conditions (see Theorem 9 in \cite{CLT} and note
that Gibbs measures for H\"{o}lder potentials on subshifts of
finite type have positive entropy).

Let us introduce the following notation, let
$\mathcal{M}_{T}(\Lambda)$ be the set of T-invariant probability
measures on $\Lambda$ and $\mathcal{M}_{\sigma^n}(\Sigma_A)$ be
the set of $\sigma^n$ invariant probability measures on
$\Sigma_A$. For $\mu\in M_{\sigma}(\Sigma_A)$ let
$h(\mu,\sigma^n)$ denote the entropy of $\mu$ with respect to
$\sigma^n$ (recall that by Abramov's Theorem
$nh(\mu,\sigma)=h(\mu,\sigma^n)$ for
$\mu\in\mathcal{M}_{\sigma}(\Sigma_A)$ ). Let
$\psi:\Sigma_A\rightarrow\R$ be defined by
$\psi(\underline{i})=\log |T_{i_1}'(\Pi(\underline{i}))|$
and for $\mu\in M_{\sigma}(\Sigma_A)$ we define the Lyapunov
exponent of the measure $\mu$ as

$$\lambda(\mu)=\int
\psi(\underline{i})\text{d}\mu(\underline{i}).$$ We will call an
ergodic measure $\mu\circ\Pi^{-1}$ {\it hyperbolic} if
$\lambda(\mu)>0$. We will use a more restricted family of measures
as well. An ergodic measure for which the support does not contain
any parabolic periodic orbits will be called {\it uniformly hyperbolic}.
Moreover we will denote the family of uniformly hyperbolic
measures by $\mathcal{M}_H(\Sigma_A)$. Obviously, any uniformly
hyperbolic measure is hyperbolic. Similarly, closed and
$T$-forward invariant subsets of $\Lambda$ that do not contain any
parabolic periodic orbits will be called {\it hyperbolic sets} and the
family of all hyperbolic sets will be denoted by $\mathcal{H}$.
The family of hyperbolic sets on which $T$ is an uniformly
expanding subshift of finite type will be denoted by
$\mathcal{H}_0$. Let
$$\alpha_{\text{min}}=\min_{\mu\in\mathcal{M}_{\sigma}(\Sigma_A)}\left\{-\frac{\int
\phi\text{d}\mu}{\lambda(\mu)}:\lambda(\mu)>0\right\}$$ and
$$\alpha_{\text{max}}=\max_{\mu\in\mathcal{M}_{\sigma}(\Sigma_A)}\left\{-\frac{\int
\phi\text{d}\mu}{\lambda(\mu)}:\lambda(\mu)>0\right\}.$$ Clearly,
$\alpha_{\max}=\infty$ if and only if $\Lambda$ contains a
parabolic periodic orbit (if: consider linear combinations of a measure
with positive Lyapunov exponent and the invariant measure supported on the
parabolic periodic orbit; only if: $-\frac 1 n S_n \phi$ is always bounded
and $\frac 1 n S_n \psi$ is bounded away from zero in the absence of
parabolic points).

We can now state our main results.
\begin{theorem}\label{main}
Let $(X,T)$ be non-uniformly expanding one-dimensional Markov map
and let $\mu$ be a weak Gibbs measure for the potential $\phi$. Then
\begin{enumerate}
\item[]
\item[1.]
$X_\alpha$ is empty for $\alpha<\alpha_{\min}$ or $\alpha
>\alpha_{\max}$.
\item[2.]
  The function $\alpha\to \dim X_\alpha$ is continuous on
$(\alpha_{\min}, \alpha_{\max})$ and right-continuous at
$\alpha_{\min}$. If $\alpha_{\max}<\infty$ then it is also
left-continuous at $\alpha_{\max}$.
\item[3.]
  For any $\alpha\in
[\alpha_{\min},\alpha_{\max}]\setminus \{\infty\}$ we have that
$$\dim
X_{\alpha}=\sup_{\mu\in\mathcal{M}_{\sigma}(\Sigma_A)}\left\{\frac{h(\mu,\sigma)}{\lambda(\mu)}:-\frac{\int
\phi\text{d}\mu}{\lambda(\mu)}=\alpha\text{ and
}\lambda(\mu)>0\right\}.$$
\item[4.]
If $\alpha_{\max}=\infty$ then $\dim
X_\alpha$ is a nondecreasing function of $\alpha$ and
$$\dim X_{\infty}=\dim\Lambda.$$
\end{enumerate}
\end{theorem}
In addition $\dim X_{\alpha}$ can be described as a supremum over
the dimension of $X_{\alpha}$ restricted to hyperbolic subsets.
\begin{theorem} \label{main3}
Under assumptions of Theorem \ref{main}, for any $\alpha\in
(\alpha_{\min},\alpha_{\max})$ we have that
\[
\dim X_\alpha = \sup_{B\in \mathcal{H}_0} \dim (B\cap
X_\alpha)=\sup_{\mu\in\mathcal{M}_H(\Sigma_A)}\left\{\frac{h(\mu,\sigma)}{\lambda(\mu)}:-\frac{\int
\phi\text{d}\mu}{\lambda(\mu)}=\alpha\right\}.
\]
\end{theorem}
Consider the family of potentials $\psi_a=a\psi+b(a)\phi$, where
$b(a)$ is the unique number for which $P(\psi_a)=0$. Like in the
hyperbolic case, the local dimension spectrum can be written as
the Lagrange-Fenchel transform of the function $b$.
\begin{theorem}\label{main2}
Under the assumptions of Theorem \ref{main}, for every
$\alpha\in(\alpha_{\min},\alpha_{\max})$ we have that
\begin{enumerate}
\item $\dim X_{\alpha}=\inf_a\{b(a)\alpha-a\}$.
\item There exists
$a\in \R$ and $\mu_a\in M_\sigma(\Sigma_A)$ such that $\mu_a$ is
an equilibrium state for $\psi_a$, $-\int \phi
d\mu_a/\lambda(\mu_a)=\alpha$ and
\[
\dim X_\alpha  = \frac {h(\mu_\alpha,\sigma)} {\lambda(\mu_\alpha)}.
\]
\item The function $\alpha\rightarrow\dim X_{\alpha}$ is concave
on $[\alpha_{\min},\alpha_{\max}]\setminus \{\infty\}$.
\end{enumerate}
\end{theorem}

Let us add some remarks, comparing our results for parabolic maps
with the usual behaviour of expanding maps. Since the space of
invariant measures is closed under the weak* topology and the
entropy is upper semi-continuous, the supremum in Theorem
\ref{main} is always achieved. However, in contrast to the
uniformly hyperbolic situation, this is not always the case with
Theorem \ref{main3}. Indeed, consider a Manneville-Pomeau map for
which an absolutely continuous, invariant, probability measure
$\nu_0$ exists. Let $\phi\equiv -\log 2$. In this situation the
Hausdorff dimension of $X_\alpha$ equals the Hausdorff dimension
of the set of points for which the Lyapunov exponent equals $\log
2/\alpha$ (compare Theorem \ref{main3} with \cite{GR}). If we let
$\alpha>\alpha_0=\log 2/\lambda(\nu_0)$ then we can deduce from
the results in \cite{GR} that $\dim X_\alpha=1$ . However, there
cannot exist any hyperbolic sets of dimension 1 (see  \cite{GR1}
), thus neither of the suprema in Theorem \ref{main3} are achieved.

Moreover, note that the supremum in Theorem \ref{main} is not
necessarily achieved among ergodic measures (i.e. dropping the
uniform hyperbolicity assumption from Theorem \ref{main3}).
Indeed, in the example above $\nu_0$ is the unique ergodic measure
of dimension 1 (\cite{Led}) and hence for $\alpha>\alpha_0$ the
supremum in Theorem \ref{main} cannot be obtained by an ergodic
measure. This is not a contradiction with Theorem \ref{main2} for
the following reason. For all $a\leq -1$, corresponding to
$\alpha\geq\alpha_0$, we have $b(a)=0$ and the Dirac measure at
the parabolic point $\delta_0$ is an ergodic equilibrium state.
For $a=-1$ we have another ergodic equilibrium state, the
absolutely continuous measure $\nu_0$ $\mu_a=\nu_0$, and all the
linear combinations of $\delta_0$ and $\nu_0$ will be
(non-ergodic) equilibrium states too. For all $\alpha>\alpha_0$
the equilibrium state provided by Theorem \ref{main3} is one of
those linear combinations, hence it is invariant but not ergodic.

For all the above results we only need the tempered distortion
property, which easily follows from our assumptions (see Lemma
\ref{distortion}). However, to prove the last result of our paper
we need something more. We refer the reader to the second section
of \cite{U} where these additional assumptions were first
introduced, we only want to mention here that whilst one can
construct examples when they are not satisfied, they hold for all
the most important examples like the Manneville-Pomeau map and the Farey
map.
\begin{theorem} \label{thm:anal}
Assume that $T$ is $C^{1+\theta}$ and our potential $\phi$ is
$C^{\theta}$ for some $\theta>0$ . We also assume that for every
parabolic point $\omega$ of period $n$ $|(T^n)'|$ is monotone on
sufficiently small one-sided neighbourhoods of $\omega$ and that
there exist constants $L>0$ and $0<\beta<\frac \theta {(1-\theta)}$
(or $\beta\in\R_+$ if $T$ is $C^2$) such that
\[
L^{-1}\leq \liminf_{x\to \omega} \frac {||(T^n)'(x)|-1|}
{|x-\omega|^\beta} \leq \limsup_{x\to \omega} \frac
{||(T^n)'(x)|-1|} {|x-\omega|^\beta} \leq L
\]
Then the function $\alpha\to\dim X_\alpha$ is real analytic on
$(\alpha_{\min},\alpha_{\max})$ with exception of at most one
point.
\end{theorem}

More precisely, this function is real analytic if and only if no
SRB measure exists, otherwise it is constant (equal to $\dim
\Lambda$) for sufficiently big $\alpha$. We comment a bit more on
this at the end of section 7.

Let us remark here that for the measure of maximal entropy its
level set for local dimension $\alpha$ coincides with the level
set for Lyapunov exponent $\frac h \alpha$. Hence, the piecewise
analyticity of the local dimension spectrum for the measure of
maximal entropy implies piecewise analyticity of the Lyapunov
spectrum. By applying the results from \cite{GR} this implies the
piecewise analyticity of $P(t\log |T'|)$ (the only place where the
analyticity will fail is at the solution to $P(t\log |T'|)=0$.)

The paper is structured as follows. In section 2 we prove some
preparatory lemmas relating the local dimension to the symbolic
dynamics. Section 3 deals with the question of relating upper
bounds for $X_{\alpha}$ to invariant measures which is a key
ingredient for the upper bound in part 3 of Theorem \ref{main} and the upper bound in Theorem \ref{main3}. Section 4
contains the proof for Theorem \ref{main}. Theorems
\ref{main3},\ref{main2},\ref{thm:anal} are then proved in sections
5,6 and 7 respectively.

\section{Symbolic dynamics and local dimension}\label{section2}

A key technique in the proof of Theorem \ref{main} will be to
relate cylinder sets in the shift space to sets in $I$. We will
denote
$$\Delta_{i_1,\ldots,i_n}=\sppan\Pi([i_1,\ldots,i_n]).$$
We will also use the notation
$\Delta_n(\underline{i})=\Delta_{i_1\ldots i_n}$. We will define
$D_n:\Sigma_A\rightarrow\R$ by
$$D_n(\underline{i})=\text{diam}(\Delta_n(\underline{i})).$$
The tempered distortion property guarantees that in most cases
$\frac 1 n \log D_n(\underline{i})$ is close to $\frac 1 n
S_n\psi(\underline{i})$.

\begin{defn} \label{distortionlem}
We say that the {\it tempered distortion} holds for a potential $f:\Sigma_A\rightarrow\R$ if there
exists a decreasing sequence $K_n(f)$ such that
$\lim_{n\rightarrow\infty}K_n(f)=0$ and for any
$\underline{i}\in\Sigma_A$ we have
$$\sup\{S_n f(\underline{i})-S_n f(\underline{j}):\underline{i},\underline{j}\in
[i_1,\ldots,i_n]\}\leq nK_n(f).$$
\end{defn}
It is clear that by compactedness any continuous potential on $\Sigma_A$ will satisfy this property.
The following important property (which allowed us, among other
things, to put a symbolic coding on $\Lambda$, as defined in the
previous section) is the statement of Lemma 2.1 in \cite{U}.

\[
\lim_{n\to\infty} \max_{\underline{i}\in \Sigma_A}
D_n(\underline{i})=0.
\]
This Lemma has some additional assumptions however the proof of the Lemma clearly still works in the current setting.
The following lemma is an immediate consequence.

\begin{lemma}\label{distortion}
For any sequence $\underline{i}\in\Sigma_A$ the
distortion of the map $T^n$ restricted to
$\Delta_n(\underline{i})$ is bounded by $e^{nK_n(\psi)}$,
independently of $\underline{i}$. In particular,
$$e^{-nK_n(\psi)}\leq D_n(\underline{i}) e^{S_n\psi(\underline{i})}\leq e^{nK_n(\psi)}.$$
\end{lemma}
We will denote
\[
\rho_n=\max(k_n, K_n(\phi), K_n(\psi)).
\]

We now consider how close $x$ can be to the endpoints of
cylinders. For a sequence $\underline{i}\in\Sigma_A$ we will
define by $Z_n(\underline{i})$ the distance between
$\Pi(\underline{i})$ and the endpoints of the $n$th level cylinder
it lies in, i.e.
$$Z_n(\underline{i})=d(\Pi(\underline{i}),\partial (\Delta_n(\underline{i}))).$$
Let
$$E=\{\underline{i}\in\Sigma:\exists N\in\N\text{ such that for }n\geq N\text{
}Z_n(\underline{i})=0\}$$ and note that this is a countable set
which contains all the points where the map $\Pi$ is not
bijective.

\begin{lemma}\label{weakavoid}
There exists $K>0$ such that for any
$\underline{i}\in\Sigma_A\backslash E$ for infinitely many $n\in
N$
$$\frac {Z_n(\underline{i})} {D_n(\underline{i})} e^{n\rho_n} \geq K.$$
\end{lemma}
\begin{proof}
Let $j_1 j_2$ and $j_3 j_4$ be the symbolic expansions of the
leftmost and rightmost second level cylinders in $I$. We note that
for $\underline{i}\notin E$ there will be infinitely many $n$ for
which $\sigma^n\underline{i}$ will start with a sequence other
than $j_1 j_2$ or $j_3 j_4$, that is,

\[
\Pi(\underline{i}) \in \Delta_{i_1,\ldots,i_n} \setminus
(\Delta_{i_1,\ldots,i_n j_1 j_2} \cup \Delta_{i_1,\ldots,i_n j_3
j_4})
\]
The triple $(\Delta_{i_1,\ldots,i_n}, \Delta_{i_1,\ldots,i_n j_1
j_2}, \Delta_{i_1,\ldots,i_n j_3 j_4})$ gets mapped onto
$(\Delta_\emptyset, \Delta_{j_1 j_2}, \Delta_{j_3 j_4})$ by the
map $T^n$. The result then follows from  Lemma \ref{distortion}.
\end{proof}

We now show that for typical points with respect to an ergodic
measure with positive entropy we have an even stronger result.
\begin{lemma}\label{strongavoid}
Let $\mu\in\mathcal{M}_{\sigma}(\Sigma_A)$ be ergodic and satisfy
$h(\mu,\sigma)>0$. We have that for $\mu$ almost all
$\underline{i}$
$$\lim_{n\to \infty} \frac{1}{n} \log
\left(\frac{Z_n(\underline{i})}{D_n(\underline{i})}\right)=0.$$
\end{lemma}
\begin{proof}
Fix such an ergodic measure $\mu$. For $\underline{i}\in\Sigma_A$
let
$$k_n(\underline{i})=\max_{k}\{i_{n+1}=i_n,\ldots,i_{n+k}=i_n\}.$$
Define $f:\Sigma_A\rightarrow\R^p$ by $f(\underline{i})=e_{i_1}$
(the $i_1$th unit vector in $\R^p$) and note that by the Birkhoff
Ergodic Theorem for $\mu$ almost all $\underline{i}$
\begin{equation} \label{eqn:erg}
\lim_{n\rightarrow\infty}\frac{S_nf(\underline{i})}{n}=\int
f\text{d}\mu
\end{equation}
and since $h(\mu)>0$ we have $\int f\text{d}\mu\in [0,1)^p$.
However, 
$$S_{n+k_n(\underline{i})}f(\underline{i})=S_nf(\underline{i})+k_n(\underline{i})e_{i_n}.$$
Thus, for the limit in \eqref{eqn:erg} to exist, we must have
$k_n(\underline{i})=\ordo{n}$. The result now follows by tempered
distortion and the fact that (because of boundedness of $|T'|$)
$$\max_{\underline{i}\in\Sigma_A}-\log D_n(\underline{i})=\Ordo{n}.$$
\end{proof}

\section{Covers of $X_{\alpha}$ and invariant measures}

In this section we relate covers of the sets $X_{\alpha}$ to
certain invariant measures. This will allow us to show that
$$\dim X_{\alpha}\leq\sup_{\mu\in\mathcal{M}(\Sigma_A)}\left\{\frac{h(\mu,\sigma)}{\lambda(\mu)}:-\frac{\int
\phi\text{d}\mu}{\lambda(\mu)}=\alpha\right\}$$
which gives the upper bound for part 3 of Theorem \ref{main}. These results will also be needed to prove Theorem \ref{main3}.
We start by giving a simple result
relating $\sigma^n$ invariant measures to $\sigma$-invariant
measures.
\begin{lemma}\label{approx}
Let $\mu\in\mathcal{M}_{\sigma^n}(\Sigma_A)$. We can find an
invariant measure $m\in\mathcal{M}_{\sigma}(\Sigma_A)$ such that
\begin{enumerate}
\item $h(m,\sigma)=\frac{h(\mu,\sigma^n)}{n}$ \item
$\lambda(m,\sigma)=\frac{\int S_n\psi\text{d}\mu}{n}$ \item $\int
\phi\text{d}m=\frac{\int S_n\phi\text{d}\mu}{n}$
\end{enumerate}
and if $\mu$ is ergodic with respect to $\sigma^n$ then $m$ is
ergodic with respect to $\sigma$.
\end{lemma}
\begin{proof}
   Let $m=\frac{1}{n}\sum_{i=0}^{n-1}\mu\circ\sigma^{-i}$. Property (1) follows
from Abramov's
   formula and properties (2) and (3) are straightforward calculations.
\end{proof}

To find good covers of $X_{\alpha}$ it is helpful to relate
$X_{\alpha}$ to sets defined by the behaviour of the ergodic ratio
of $\psi$ and $\phi$. If we fix $\alpha\in\R$ then we can use the
results in the previous section to obtain the following lemma.
\begin{lemma}
If $\underline{i}\in\Sigma_A\backslash E$, $\Pi\underline{i}=x$
and $x\in X_{\alpha}$ then
$$\liminf_{n\rightarrow\infty}-\frac{S_n\phi(\underline{i})}{S_n\psi(\underline{i})}=\alpha.$$
\end{lemma}
\begin{proof}
We combine Lemma \ref{weakavoid} and the tempered distortion
property to see that there exists a subsequence $\{n_m\}_{m\in\N}$
such that for each $m\in\N$
$$B(x,K^{-1}e^{-2n_m\rho_{n_m}-S_{n_m}\psi(\underline{i})})\subseteq\Delta_n(\underline{i}).$$
Thus for $\epsilon>0$ and $m$ sufficiently large we have that
\begin{eqnarray*}
\nu(\Delta_{n_m}(\underline{i}))&\geq&
\nu(B(x,K^{-1}e^{-2n_m\rho_{n_m}
-S_{n_m}\psi(\underline{i})}))\\
&\geq&
(K^{-1}e^{-2n_m\rho_{n_m}-S_{n_m}\psi(\underline{i})})^{\alpha+\epsilon}.
\end{eqnarray*}
At the same time,

\[
\nu(\Delta_{n_m}(\underline{i}))\leq K e^{n_m \rho_{n_m} + S_{n_m}
\phi(\underline{i})}.
\]
Note that $-\frac 1 n S_n \phi$ is uniformly bounded away from
zero. This implies that $\frac 1 n S_n\psi(\underline{i})$ is
bounded away from zero as well, dominating $\rho_n$. Hence,

\[
e^{S_{n_m}\phi(\underline{i})} \geq  K^{-\alpha -\epsilon -1}
e^{-(2\alpha+ 2\epsilon+1)n_m \rho_{n_m}} \cdot
e^{-(\alpha+\epsilon) S_{n_m} \psi(\underline{i})}
\]
proves the inequality in one direction. The other direction
follows from
\[
B(x,Ke^{n\rho_n-S_n\psi(\underline{i})})\supseteq\Delta_n(\underline{i})
\]
(it holds for all $n$ by the tempered distortion property) in a
similar way.
\end{proof}

So if we let

\[
Y_{\alpha}= \left\{\underline{i}\in\Sigma_A:
\liminf_{n\rightarrow\infty}-\frac{S_n\phi(\underline{i})}{S_n\psi(\underline{i})}=\alpha\right\}
\]
then we can deduce that
$$\dim X_{\alpha}\leq \dim\Pi Y_{\alpha}.$$
To obtain the upper bound for Theorem \ref{main} we will calculate
an upper bound for the dimension of $\Pi Y_{\alpha}$. Let

\[
Y_{\alpha,\epsilon}=\left\{\underline{i}\in\Sigma_A:
\alpha-\epsilon \leq
-\frac{S_n\phi(\underline{i})}{S_n\psi(\underline{i})}\leq\alpha+\epsilon\text{
for infinitely many }n\in\N\right\}.
\]
Let

\[
Y_{n,\alpha,\epsilon}=\left\{[i_1,\ldots,i_n]:\exists\underline{i}\in
[i_1,\ldots,i_n]\text{ with
}\left(-\frac{S_n\phi(\underline{i})}{S_n\psi(\underline{i})}\right)\in(\alpha-\epsilon,
\alpha+\epsilon)\right\}
\]
and note that for all $N\in \N$ we
have $\cup_{n\geq N}Y_{n,\alpha,\epsilon}\supseteq
Y_{\alpha,\epsilon}$. We now define $s_n$ to satisfy
\begin{equation} \label{noidea}
\sum_{[i_1,\ldots,i_n]\in
Y_{n,\alpha,\epsilon}}(D_n(\tilde{\underline{i}}))^{s_n}=1
\end{equation}
where $\tilde{\underline{i}}$ denotes some element of the cylinder
$[i_1,\ldots,i_n]\in Y_{n,\alpha,\epsilon}$. Let
$$s=\limsup_{n\rightarrow\infty}s_n.$$

As $-\frac 1 n S_n\phi$ is uniformly bounded away from zero, for
all cylinders $[i_1,\ldots,i_n]\in Y_{n,\alpha,\epsilon}$ we have
$-\frac 1 n \log D_n(\underline{i})$ also bounded away from zero
for $\underline{i}\in [i_1,\ldots,i_n]$. It is thus clear that for
any $\delta>0$ there exists $N>0$ such that
\begin{equation} \label{noidea2}
\sum_{n>N}\sum_{Y_{n,\alpha,\epsilon}}(D_n(\tilde{\underline{i}}))^{s+\delta}<1.
\end{equation}
Consider now for increasing $N$ the unions $\bigcup_{n>N}
Y_{n,\alpha,\epsilon}$, each of them forms a covering for $\Pi
Y_{\alpha,\epsilon}$. By \eqref{noidea2}, it follows that
$$\dim\Pi Y_{\alpha,\epsilon}\leq s.$$

We now want to relate this value $s$ to the entropy and Lyapunov
exponent of an invariant measure. To do this we introduce a class
of measures which will not be $\sigma$-invariant but will be
$\sigma^m$ invariant for some $m$. Let
$$C_n=\{[i_1,\ldots,i_n]:[i_1,\ldots,i_n]\cap\Sigma_A\neq\emptyset\}$$
and $q:C_n\rightarrow [0,1]$ satisfy $\sum_{[i_1,\ldots,i_n]\in
C_n}q(C_n)=1$. Let $k(n)$ be the smallest integer such that for
any $\underline{i}^n,\underline{j}^n\in C_n$ there exists
$\omega(\underline{i},\underline{j})\in C_k$ such that
$$[i_1,\ldots,i_n,\omega_1,\ldots,\omega_k,j_1,\ldots,j_n]\cap\Sigma_A\neq\emptyset$$
and
$$\inf\{\exp(S_{n+k}\psi(\underline{i}):\underline{i}\in [i_1,\ldots,i_n,\omega_1,\ldots,\omega_k])\}>1.$$
We can then use our function $q$ to define a $\sigma^{n+k}$ invariant measure
$\mu_q$. We define the measure on cylinders of level $ln+(l-1)k$ by setting
$$\mu_q([\underline{i}_1^{n}\omega(\underline{i}_1^{n},\underline{i}_2^{n})
\underline{i}_2^{n}\cdots\omega(\underline{i}_{l-1}^{n},\underline{i}_l^{n}),
\underline{i}_l^{n}])=\prod_{m=1}^l q(\underline{i}_m^{n})$$ and
the measure of all cylinders of level $ln+(l-1)k$ not of this form
to be $0$. We will write the space of all such measures as
$M_{n}(\Sigma_A)$. The following Lemma describes the behaviour of
these measures.

\begin{lemma} \label{muq}
The family of measures $M_{n}(\Sigma_A)$ satisfies the following:
\begin{enumerate}
\item
If $\mu_q\in M_n(\Sigma_A)$ then $\mu_q$ is $\sigma^{n+k}$-invariant and ergodic.
\item
For $\mu_q\in M_n(\Sigma_A)$ we have that
$$h(\mu_q,\sigma^{n+k})=-\sum_{[i_1,\ldots,i_n]\in C_n}q(i_1,\ldots,i_n)\log q(i_1,\ldots,i_n)$$
with the usual convention $0\log 0=0$. \item For any $\mu_q\in
M_n(\Sigma_A)$ the set $\Pi(\text{supp }\mu_q)$ will not contain
any parabolic points. In particular the measure
$\nu_q=\left(\frac{1}{n+k}\sum_{i=0}^{n+k-1}\mu_q\circ\sigma^{-i}\right)\circ\Pi^{-1}$
will be uniformly hyperbolic. \item The measures $\mu_q\in
M_n(\Sigma_A)$ vary continuously, with $q$, in the weak* topology.
\item We have
\[
\left|\int \frac 1 {n+k} S_{n+k}\psi d\mu_q -\frac 1 n
\sum_{\underline{i}^n\in C_n} q(\underline{i}^n)
S_n\psi(\underline{i})\right| \leq \frac {kL} {n+k} +\rho_n
\]
and
\[
\left|\int \frac 1 {n+k} S_{n+k}\phi d\mu_q -\frac 1 n
\sum_{\underline{i}^n\in C_n} q(\underline{i}^n)
S_n\phi(\underline{i})\right| \leq \frac {kL} {n+k} + \rho_n
\]
where $\underline{i}$ is an arbitrary point from
$[\underline{i}^n]$ and $L=\max(\sup |\phi|, \sup \psi)$.
\end{enumerate}
\end{lemma}
\begin{proof}
Let $n\in\N$ and $q$ be a function defining a measure, $\mu_q$, in
$M_n(\Sigma_A)$. This measure is going to be defined as a
$\sigma^n$ invariant Bernoulli measure, that is, we take a
Bernoulli measure defined on $(\{1,\ldots,p\}^n)^\N$ (defined by
describing its values on each cylinder $[i_1,\ldots, i_n]$) and
then transport this measure to $\Sigma$ by the usual
identification. We set $\eta_q$ by
$$\eta_q([i_1,\ldots,i_n])=q([i_1,\ldots,i_n])\text{ if }[i_1,\ldots,i_n]\cap\Sigma_A\neq\emptyset$$
and
$$\eta_q([i_1,\ldots,i_n])=0\text{ if }[i_1,\ldots,i_n]\cap\Sigma_A=\emptyset.$$
We can build a map, $\gamma:\text{supp}\mu_q\rightarrow\Sigma$ by
$$\gamma(\underline{i}_1^{n}\omega(\underline{i}_1^{n},\underline{i}_2^{n})
\underline{i}_2^{n}\cdots\omega(\underline{i}_{l-1}^{n},\underline{i}_l^{n})
\underline{i}_l^{n}\cdots)=\underline{i}_1^{n}\underline{i}_2^{n}\cdots
,\underline{i}_l^{n}\cdots$$ which maps $\mu_q$ onto $\eta_q$. In
fact this is an isomorphism of the dynamical systems $(\Sigma,
\sigma^{n+k}, \mu_q)$ and $(\Sigma, \sigma^n, \eta_q)$. Parts 1
and 2 follow immediately. Part 3 follows from the definition of
$\omega$. Part 4 follows from the definition of $\gamma$ and the
weak* continuity of Bernoulli measures with respect to the
generating vector.

To see part 5 we note that
\[
\left|\int \frac 1 n S_n\psi d\mu_q -\frac 1 n
\sum_{\underline{i}^n\in C_n}
q(\underline{i}^n) S_n\psi(\underline{i})\right| \leq \rho_n
\]
and
\[
\left|\int \frac 1 n S_n\phi d\mu_q -\frac 1 n
\sum_{\underline{i}^n\in C_n}
q(\underline{i}^n) S_n\phi(\underline{i})\right| \leq \rho_n
\]
follow from tempered distortion. The assertion now follows by direct calculation.
\end{proof}

The importance of this family of measures is that certain measures in
this family can be related to $s$, the upper bound for $\dim X_{\alpha}$.

\begin{lemma} \label{density}
We can find a sequence of  measures
$\mu_n\in M_n(\Sigma_A)$ such that
$$\alpha - \epsilon \leq \liminf_{n\rightarrow\infty}-\frac{\int
S_{n}\phi\text{d}\mu_n}{\int S_{n}\psi\text{d}\mu_n}\leq
\limsup_{n\rightarrow\infty}-\frac{\int
S_{n}\phi\text{d}\mu_n}{\int S_{n}\psi\text{d}\mu_n}
\leq\alpha+\epsilon,
$$
and
$$\lim_{n\rightarrow\infty}\left( \frac{h(\mu_n,\sigma^{n+k})}{\int S_{n}\psi\text{d}\mu_n} - s_n\right) =0.$$
Those measures are generated with $q$ supported in $Y_{n,\alpha, \epsilon}$.
\end{lemma}
\begin{proof}
Given $\underline{i}^n\in C_n$, we let
$q(\underline{i}^n)=(D_n(\underline{i}))^{s_n}$ if
$\underline{i}^n\in Y_{n,\alpha, \epsilon}$ and
$q(\underline{i}^n)=0$ otherwise. We consider the corresponding
measure $\mu_q\in M_n$.
Let us estimate  the relevant parameters for the measure $\mu_q$.
As $\frac 1 n S_n\phi$ is bounded away from zero, the definition
of $Y_{n,\alpha, \epsilon}$ implies that $\frac 1 n S_n \psi$ is
bounded away from zero as well, hence by Lemma \ref{muq},
statement 5

\[
\alpha - \epsilon - o(1) \leq -\frac {\int S_{n+k}\phi d\mu_q}
{\int S_{n+k}\psi d\mu_q} \leq \alpha+\epsilon+o(1),
\]
hence

\begin{equation} \label{somename3}
\alpha - \epsilon - o(1) - O(1/n) \leq -\frac {\int S_n\phi
d\mu_q} {\int S_n\psi d\mu_q} \leq \alpha+\epsilon+o(1)+O(1/n),
\end{equation}
which proves the first part of the assertion. We also have that

\[
h(\mu_q,\sigma^{n+k})=-s_n \sum_{\underline{i}^n\in
Y_{n,\alpha,\epsilon}} (D_n(\underline{i}^n))^{s_n} \log
D_n(\underline{i}^n).
\]

On the other hand, by combining Lemma \ref{distortion} and Lemma
\ref{muq} statement 5 we can deduce that
\begin{equation} \label{somename2}
\left|\int \frac 1 n S_n\psi d\mu_q + \frac 1 n
\sum_{\underline{i}^n\in Y_{n,\alpha,\epsilon}}
(D_n(\underline{i}^n))^{s_n} \log D_n(\underline{i}^n)\right| \leq
\left|\int \psi d\mu_q - \frac 1 n \sum_{\underline{i}^n\in
Y_{n,\alpha, \epsilon}} q(\underline{i}^n)
S_n\psi(\underline{i})\right|+\rho_n \to 0.
\end{equation}
  As the left hand side of \eqref{somename2} is
a difference of two terms, the first of which is bounded away from
zero by \eqref{somename3}, the second term is bounded as well. The second part of assertion follows.
\end{proof}
By combining this lemma with Lemma \ref{approx} we get an immediate corollary.

\begin{corollary} \label{ergodicdensity}
We can find a sequence of ergodic measures
$\mu_n\in\mathcal{M}_\sigma(\Sigma_A)$ such that
$$\alpha-\epsilon \leq\lim_{n\rightarrow\infty}-\frac{\int
\phi\text{d}\mu_n}{\lambda(\mu_n)}\leq\alpha+\epsilon$$ and
$$\limsup_{n\rightarrow\infty}\frac{h(\mu_n,\sigma)}{\int \psi\text{d}\mu_n}\geq s.$$
Moreover these measures are uniformly hyperbolic.
\end{corollary}
\begin{proof}
We consider $\mu_q\in M_n(\Sigma_A)$ and choose a subsequence for
which the limit of $\int \phi d\mu_{q_i}/\int \psi d\mu_{q_i}$
exists. By Lemma \ref{approx} we can find an ergodic measure
$\tilde{\mu}\in \mathcal{M}_{\sigma}(\Sigma_A)$ such that
$h(\tilde{\mu})=\frac{1}{n+k} h(\mu_q,\sigma^{n+k})$, $\int
\psi\text{d}\tilde{\mu}=\frac{1}{n+k}\int
S_{n+k}\psi\text{d}\mu_q$ and $\int
\phi\text{d}\tilde{\mu}=\frac{1}{n+k}\int
S_{n+k}\phi\text{d}\mu_q$. Since $k$ is independent of $n$ the
result easily follows from Lemma \ref{density}.
\end{proof}

In the proof of Theorem \ref{main3} we will need the sequence
$s_m$ to be convergent. This can be deduced from Lemma
\ref{density}.

\begin{corollary} \label{sn}
In the above construction we have that
$s=\limsup_{n\rightarrow\infty}s_n=\lim_{n\rightarrow\infty}s_n$.
In other words the sequence $s_n$ is convergent.
\end{corollary}
\begin{proof}
Consider the measure $\nu_n = \mu_n \circ \Pi^{-1}$, where $\mu_n$
is defined in Corollary \ref{ergodicdensity}. As it is ergodic,
its Hausdorff dimension equals its local dimension at a typical
point. By Lemma \ref{strongavoid} at a typical point its local
dimension is equal to its symbolic local dimension $\lim \log
\nu_n(\Delta_n(\underline{i}))/\log D_n(\underline{i})$. By
Birkhoff Ergodic Theorem, there exists some $N$ such that for all
$m>N$ the $m$-th level cylinders $[i_1,\ldots,i_m]$ for which

\[
D_m(\underline{i})^{\dim \nu_n - \epsilon} >
\mu_n([i_1,\ldots,i_m])
\]
have total measure $\mu_n$ greater than 2/3.

As $\mu_n$ is ergodic, for $\mu_n$-almost every $\underline{i}$
$[i_1,\ldots,i_m] \in Y_{m,\alpha, 2\epsilon}$ for $m$ big enough.
This implies that there exists $N$ such that for all $m>N$ the
cylinders from $Y_{m,\alpha,2\epsilon}$ have total measure $\mu_n$
greater than 2/3.

Combining those two statements, we see that for all $m$ big enough

\begin{equation} \label{ghghg}
\sum_{[i_1,\ldots,i_m]\in Y_{m,\alpha,2\epsilon}}
(D_m(\underline{i}))^{\dim \nu_n - \epsilon} > 1/3.
\end{equation}
Since the diameters of the cylinders from $Y_{m,\alpha,
2\epsilon}$ are exponentially small \eqref{ghghg} implies
that

\[
s_m \geq \dim \nu_n -\epsilon -O(1/m)
\]
for all sufficiently big $m$. At the same time, by \cite{HR}

\[
\dim \nu_n = \frac {h(\mu_n, \sigma)} {\int \psi d\mu_n}
\]
which, by Corollary \ref{ergodicdensity}, can be chosen
arbitrarily close to $s$.
\end{proof}

\section{Proof of Theorem \ref{main}}
The first statement follows immediately from Corollary
\ref{ergodicdensity}. The second and fourth statements are going
to be obtained in the course of the proof of the main, third
statement. We will first prove the easy upper bound.

It follows immediately from Corollary \ref{ergodicdensity} that
for any $\epsilon>0$
$$\dim\Pi
Y_{\alpha}\leq\sup_{\mathcal{M}_H(\Sigma_A)}\left\{\frac{h(\mu,\sigma)}{\lambda(\mu)}:\alpha-\epsilon
\leq -\frac{\int
\phi\text{d}\mu}{\lambda(\mu)}\leq\alpha+\epsilon\text{ and
}\lambda(\mu)>0\right\}.$$
Note that since $\mathcal{M}_H(\Sigma_A)\subset\mathcal{M}_{\sigma(\Sigma_A)}$ we also have the same inequality if the supremum is taken over all invariant measures,
$$\dim\Pi
Y_{\alpha}\leq\sup_{\mathcal{M}_\sigma(\Sigma_A)}\left\{\frac{h(\mu,\sigma)}{\lambda(\mu)}:\alpha-\epsilon
\leq -\frac{\int
\phi\text{d}\mu}{\lambda(\mu)}\leq\alpha+\epsilon\text{ and
}\lambda(\mu)>0\right\}.$$

  To complete the proof of the upper
bound we need to show the supremum over invariant measures
$$\sup_{\mu\in\mathcal{M}_{\sigma}(\Sigma_A)}\left\{\frac{h(\mu,\sigma)}{\lambda(\mu)}:-\frac{\int
\phi\text{d}\mu}{\lambda(\mu)}=\alpha\text{ and
}\lambda(\mu)>0\right\}$$ varies continuously with $\alpha$.

The supremum is an upper semi-continuous function of $\alpha$
because of the upper semi-continuity of entropy (see Theorem 8.2 in \cite{W}) and the continuity
of the Lyapunov exponent. We now fix $\alpha<\beta\in
[\alpha_{\min},\alpha_{\max}]\backslash\{\infty\}$ and let
$\mu_\alpha,\mu_{\beta}\in\mathcal{M}_{\sigma}(\Sigma_A)$ satisfy
that $-\frac{\int
\phi\text{d}\mu_{\alpha}}{\lambda(\mu_{\alpha})}=\alpha$,
$-\frac{\int
\phi\text{d}\mu_{\beta}}{\lambda(\mu_{\beta})}=\beta$,
$\frac{h(\mu_{\alpha},\sigma)}{\lambda(\mu_{\alpha})}=\dim X_{\alpha}$
and $\frac{h(\mu_{\beta},\sigma)}{\lambda(\mu_{\beta})}=\dim X_{\beta}$
(those measures exist because $\alpha <\beta <\infty$ implies
$\lambda(\mu_\alpha), \lambda(\mu_\beta)\geq c(\beta)>0$). To show
that the function is right lower semi-continuous at $\alpha$ and
left lower semi-continuous at $\beta$ we simply consider convex
combinations of $\mu_\alpha$ and $\mu_\beta$.

We now turn to the lower bound. Initially we will prove the lower
bound in the case where there exist hyperbolic measures with
dimension arbitrarily close to that of $\Lambda$. To complete the
proof we will then need to show that $\dim X_{\infty}=\dim\Lambda$
still holds in the case when such hyperbolic measures are not
known to exist. Note that if $T$ is $C^{1+\theta}$ then such
measures always exist and that it is unknown whether this is the
case if $T$ is merely $C^1$.

We begin with the following lemma:
\begin{lemma}\label{symbolic}
Let $\underline{i}\in\Sigma_A\backslash E$, $x=\Pi\underline{i}$
and $\alpha\in\R\cup\{\infty\}$. If
$$\lim_{n\rightarrow\infty}-\frac{S_n\phi(\underline{i})}{S_n\psi(\underline{i})}=\alpha$$
and
\begin{equation}\label{middle}
\lim_{n\to \infty} \frac{1}{n} \log \left(\frac {Z_n(\underline{i})}
{D_n(\underline{i})}\right)=0
\end{equation}
  then
$$\lim_{r\rightarrow 0}\frac{\log\nu(B(x,r))}{\log r}=\alpha.$$
\end{lemma}
\begin{proof}
Let $\underline{i}\in\Sigma_A\backslash E$ satisfy the hypothesis
in the Lemma with $\alpha$ finite and let $x=\Pi\underline{i}$.
Fix $\epsilon,r>0$ and choose $n$ such that
$$D_n(\underline{i})<r\leq D_{n-1}(\underline{i}).$$
It follows that
$$B(x,r)\supset \Delta_n(\underline{i})$$
and if $r$ is chosen to be sufficiently small then
$$\nu(B(x,r))\geq (D_{n-1}(\underline{i}))^{\alpha+\epsilon}.$$
Furthermore, by property (\ref{middle}) for $r$ small enough
$$B(x,r^{1+\epsilon})\subset \Delta_n(\underline{i})$$
and hence
$$\nu(B(x,r^{1+\epsilon}))\leq (D_n(\underline{i}))^{\alpha-\epsilon}\leq
r^{\alpha-\epsilon}.$$ Since $\epsilon$ can be chosen to be
arbitrarily small, it follows that
$$\lim_{r\rightarrow 0}\frac{\log\nu(B(x,r))}{\log r}=\alpha.$$
For $\alpha=\infty$ we fix $\epsilon,r>0$ and choose $n$ in the
same way. For any $\beta>0$ if $r$ is sufficiently small then we
have
\begin{eqnarray*}
\nu(B(x,r^{1+\epsilon}))&\leq&\tilde{\nu}([i_1,\ldots,i_n])\\
&\leq&(D_n(\underline{i}))^\beta\leq r^{\beta}
\end{eqnarray*}
and the result follows.
\end{proof}

By Lemma \ref{strongavoid}, the assumptions of Lemma
\ref{symbolic} are satisfied for almost all points for any ergodic
measure $\mu$ on $\Sigma_A$, with

\[
\alpha = - \frac {\int \phi d\mu} {\lambda(\mu)}.
\]
Thus, almost all points for measure $\mu \circ \Pi^{-1}$ will
belong to $X_\alpha$, and by \cite{HR} the Hausdorff dimension of
this set is $h(\mu,\sigma)/\lambda(\mu)$. Hence,

\begin{equation} \label{ergodic}
\dim X_\alpha \geq \sup \left\{\frac {h(\mu,\sigma)} {\lambda(\mu)}: \mu\
\text{ergodic}, -\frac {\int \phi d\mu} {\lambda(\mu)}
=\alpha\right\}.
\end{equation}

However, what we need to prove is the lower bound given by
supremum over all invariant measures, not only ergodic measures in
\eqref{ergodic}. The first case is $\alpha\in
\{\alpha_{\text{min}}, \alpha_{\text{max}}\}$, and we will prove
that for such $\alpha$ suprema over ergodic measures and over
invariant measures are equal.

For $\alpha=\alpha_{\text{min}}$ we consider any invariant measure
$\mu$ such that
$$-\frac{\int \phi\text{d}\mu}{\int \psi\text{d}\mu}=\alpha_{\text{min}}.$$
For any measure $\mu_i$ in the ergodic decomposition of $\mu$ it
follows that
$$-\frac{\int \phi\text{d}\mu_i}{\int \psi\text{d}\mu_i}=\alpha_{\text{min}}$$
since otherwise there would exist an ergodic measure contradicting
the definition of $\alpha_{\min}$. Moreover at least one measure
in the ergodic decomposition must satisfy
$\frac{h(\mu_i,\sigma)}{\lambda(\mu_i)}\geq\frac{h(\mu,\sigma)}{\lambda(\mu)}$.
This completes the proof of the lower bound for
$\alpha=\alpha_{\text{min}}$.

As the same proof works for $\alpha=\alpha_{\max}<\infty$, there
are two cases left: $\alpha\in(\alpha_{\min},\alpha_{\max})$ and
$\alpha=\alpha_{\max}=\infty$. This time we will not be able to
find a good measure for the given $\alpha$, but we will find them
for some close $\alpha$. In fact, the measures we are going to
find will be Gibbs measures, hence in particular ergodic. We will
then use them to construct certain big set and then prove that it
is contained in $X_\alpha$. The beginning parts of the proofs
differ in both cases but the subsequent argument is the same.

Let $\mu$ be an invariant measure satisfying $-\int \phi
d\mu/\lambda(\mu)=\alpha\in (\alpha_{\min},\infty)$. By Lemma 4.2
from \cite{O}, there exists a sequence of Gibbs measures $\mu_i$
weakly converging to $\mu$, such that $h(\mu_i,\sigma)\to
h(\mu,\sigma)$. Note these are not just weak Gibbs measures but Gibbs measures in the usual sense and thus ergodic. The weak convergence implies the convergence of
$\int \phi$ and of $\lambda$ which means that
$h(\mu_i,\sigma)/\lambda(\mu_i)\to h(\mu,\sigma)/\lambda(\mu)$ and
so the Hausdorff dimension of the measures $\mu_i \circ \Pi^{-1}$
converges to this limit.

In the case where $\alpha=\infty$ we will let
$$\eta=\sup_{\mu\in\mathcal{M}_{\sigma}(\Sigma_A)}\left\{\frac{h(\mu,\sigma)}{\lambda(\mu)}:\lambda(\mu)>0\right\}.$$
Our initial goal here will be to prove that $\dim X_\infty\geq
\eta$. By lemma 4.2 from \cite{O} it follows that for any
$\epsilon>0$ we can find an ergodic measure $\mu_1$ such that
$\dim\mu_1\geq \eta-\epsilon$. If we take a parabolic invariant
measure $\mu_2$ and consider invariant measures of the form
$q\mu_1+(1-q)\mu_2$ for $0<q<1$ then it is clear that
$$\frac{h(q\mu_1+(1-q)\mu_2,\sigma)}{\lambda(q\mu_1+(1-q)\mu_2)}=\frac{h(\mu_1,\sigma)}{\lambda(\mu_1)}$$
and that $\lim_{q\rightarrow 0}\lambda(q\mu_1+(1-q)\mu_2)=0$. Thus
again applying Lemma 4.2 from \cite{O} to these measures we get a
sequence of Gibbs measures $\mu_j$ such that
$$\lim_{j\rightarrow\infty}\frac{h(\mu_j,\sigma)}{\lambda(\mu_j)}\geq\eta-\epsilon$$
and
$$\lim_{j\rightarrow\infty} -\frac{\int \phi d\mu_j}{\lambda(\mu_j)}=\infty.$$

Hence, in both cases we have a sequence of Gibbs measures $\mu_i$
for which $-\int \phi d\mu_i/\lambda(\mu_i)\to \alpha$ and we want
to prove that
\[
\dim X_\alpha \geq \liminf \frac {h(\mu_i,\sigma)}
{\lambda(\mu_i)}.
\]
The aim will be to construct a (non-invariant) measure
$\tilde{\mu}$ such that
$\lim_{n\rightarrow\infty}\frac{1}{n}S_n\phi(\underline{i})=\lim_{n\rightarrow\infty}
\int \phi d\mu_n$,
$\lim_{n\rightarrow\infty}\frac{1}{n}S_n\psi(\underline{i})=\lim_{n\rightarrow\infty}
\lambda(\mu_n)$ for $\tilde{\mu}$-almost all $\underline{i}$,
$\dim \tilde{\mu}\geq \lim_{n\rightarrow\infty} \dim \mu_n$ and
the statement of Lemma \ref{strongavoid} is satisfied. Here we
assume (by restricting to a subsequence of $\{\mu_n\}$ if
necessary) that the limits of $\int \phi d\mu_n$, $\int \psi
d\mu_n$ exist. The measure $\tilde{\mu}$ will be constructed using
the approach from section 5 of \cite{GR}.

Let $\{m_i\}$ be an increasing sequence of integers, $m_0=0$. We
will define a new measure $\tilde{\mu}$ (w-measure in the
terminology of \cite{GR}) inductively: $\tilde{\mu}(\Lambda)=1$
and
\[
\tilde{\mu}(\Delta_{\underline{i}^{m_i}\underline{j}^{m_{i+1}-m_i}})=c_{i+1}(\underline{i}^{m_i})
\tilde{\mu}(\Delta_{\underline{i}^{m_i}})
\Pi_*\mu_{i+1}(\Delta_{\underline{j}^{m_{i+1}-m_i}})
\]
where $c_{i+1}(\underline{i}^{m_i})$ are normalizing constants. In
other words, this measure is concentrated on points whose
trajectory for some time $m_1$ is distributed according to measure
$\mu_1$, then for time $(m_2-m_1)$ it is distributed according to
$\mu_2$ and so on.

The properties of measures of this type were checked in \cite{GR}.
First, as proved in \cite{GR}, Proposition 9, the Hausdorff
dimension of $\mu$ is not smaller than the lower limit of
Hausdorff dimensions of $\mu_i$ provided that $m_i$ grow quickly
enough. Furthermore, there exists a sequence $\{K_i\}$ with each
$K_i$ depending only on $\mu_i$ such that for each $m\in
(m_i,m_{i+1})$
\begin{equation} \label{eqn:di}
K_{i+1}^{-1} \leq \frac
{\tilde{\mu}(\Delta_{\underline{i}^{m_i}\underline{j}^{m-m_i}})}
{\tilde{\mu}(\Delta_{\underline{i}^{m_i}})
\mu_{i+1}(\Delta_{\underline{j}^{m-m_i}})} \leq K_{i+1}.
\end{equation}

An immediate consequence of \eqref{eqn:di} is that for any bounded
potential (in particular, $\phi$ or $\psi$), if $m_i$ grows
quickly enough then its Cesaro average at a $\tilde{\mu}$ typical
point is equal to $\lim \int \phi d\mu_i$ (see Proposition 9 in
\cite{GR} again). In our setting this means that,
\begin{equation} \label{eqn:xprop}
\frac {S_n \phi(\underline{i})} {S_n \psi(\underline{i})} \to
\alpha
\end{equation}
$\tilde{\mu}$-almost everywhere.

To finish  we need to use this to prove that
$\tilde{\mu}(X_\alpha)=1$ which we will do by showing the
statement of Lemma \ref{strongavoid} is valid for $\tilde{\mu}$
which will allow us to apply Lemma \ref{symbolic}. Indeed, all the
$\mu_i$ satisfy Lemma \ref{strongavoid}. Hence, for every
$\epsilon$ there exists a sequence $\{C_i\}_{i\in\N}$ such that
\[
\mu_i(\{\underline{j}\in\Sigma_A: \exists k \text{ where }
Z_k(\underline{j})<C_i D_k(\underline{i}) e^{-k
\epsilon}\})<K_{i+1}^{-1}\cdot 2^{-i}.
\]
Due to \eqref{eqn:di}, this implies that for every $i$ and
$\underline{i}^{m_i}$, inside every cylinder
$[i_1,\ldots,i_{m_i}]$ the relative $\tilde{\mu}$ measure of
points $\underline{j}$ for which there exists $k\in(m_i, m_{i+1}]$
such that
\begin{equation} \label{eqn:diam}
Z_k(\underline{j})< C_{i+1} D_k(\underline{j}) e^{-k\epsilon}
\end{equation}
is not greater than $2^{-i-1}$. By the Borel-Cantelli Lemma,
\eqref{eqn:diam} is satisfied only finitely many times for
$\tilde{\mu}$-almost any $\underline{j}$.

Note here that $D_k(\underline{j})$ is decreasing exponentially
fast with $k$ for $\mu_i$-almost every $\underline{j}$ (but the
rate depends on $i$). Thus, if we choose $\{m_i\}$ to be
increasing quickly enough, we will get
\[
\liminf_{k\rightarrow\infty} \frac 1 k \log\left( \frac
{Z_k(\underline{j})} {D_k(\underline{j})}\right) \geq-2\epsilon
\]
$\tilde{\mu}$-almost everywhere. As this holds for any
$\epsilon>0$  we know that
$$\lim_{k\rightarrow\infty} \frac 1 k \log \left(\frac {Z_k(\underline{j})}
{D_k(\underline{j})}\right)=0$$ and so we can use Lemma
\ref{symbolic}  to deduce that $\tilde{\mu}(X_{\alpha})=1$. We
have shown that for $\alpha\in (\alpha_{\min},\alpha_{\max})$ we
have that
$$\dim X_{\alpha}\geq\sup_{\mu\in\mathcal{M}_{\sigma}(\Sigma_A)}\left\{\frac{h(\mu,\sigma)}{\lambda(\mu)}:-\frac{\int
\phi\text{d}\mu}{\lambda(\mu)}=\alpha\text{ and
}\lambda(\mu)>0\right\}$$ and if $\alpha_{\max}=\infty$ then
$$\dim X_{\infty}\geq\eta.$$
To complete the proof of the lower bound we now need to address
the possible case, where $\eta<\dim\Lambda$.

We start by looking at the local Lyapunov exponents of points. Let
\[
Z=\left\{\underline{i}:\limsup_{n\rightarrow\infty}\frac{S_n\psi(\underline{i})}{n}=0\right\}
\]
\begin{lemma}\label{nonzerolyap}
We have that
$$\dim(\Lambda \setminus \Pi Z)\leq \eta$$
and thus if $\eta<\dim\Lambda$ then
$$\dim\Pi Z=\dim\Lambda.$$
\end{lemma}
\begin{proof}
The first part of this Lemma can easily be deduced from the
previous section: the argument estimating the dimension of $\Pi
Y_{\alpha}$ also shows that
\begin{equation} \label{ya}
\dim \bigcup_{\alpha\in (q_1,q_2)} \Pi Y_\alpha \leq
\sup_{\mu\in\mathcal{M}_{\sigma}(\Sigma_A)}\left\{\frac{h(\mu,\sigma)}{\lambda(\mu)}:q_1-\epsilon\leq
-\frac{\int \phi\text{d}\mu}{\lambda(\mu)}\leq q_2+\epsilon\text{
and }\lambda(\mu)>0\right\}.
\end{equation}
The second part then follows immediately.
\end{proof}
We now look at points which may have zero local Lyapunov exponent
but not have infinite local dimension. Let
$$L_{\beta}=\Pi\left\{\underline{i}\in Z:Z_n(\underline{i})\leq e^{-n\beta}\text{ infinitely often}\right\}.$$
It is clear that if $\underline{i}\in Z$ but
$\Pi(\underline{i})\notin L_{\beta}$ for any $\beta$ then
$\Pi(\underline{i})\in X_{\infty}$. Thus to complete the proof it
suffices to show that $\dim L_{\beta}$ is bounded away from $\dim
\Lambda$ for all $\beta>0$.
\begin{lemma}
For any $\beta>0$ we have that $\dim L_{\beta}=0$.
\end{lemma}
\begin{proof}
Fix $\delta>0$ and let $N_0$ be large enough such that $\rho_n <
\delta$ for $n\geq N_0$. Given $n\geq N_0$ let
$$C_n=\{[i_1,\ldots,i_n]:\exists\underline{i}\in [i_1,\ldots,i_n]\text{ where }S_n\psi(\underline{i})\leq \delta n\}.$$
It follows that the diameter of the projection of each cylinder in
$C_n$ is at least $e^{-2\delta n}$ and so there can be at most
$e^{2\delta n}$ cylinders in the set $C_n$. Furthermore we have
that
$$Z\subset \bigcap_{N>N_0} \bigcup_{n\geq N}\bigcup_{I\in C_n}\Pi I.$$

Given $I\in C_n$, $\spa\, \Pi I=[x,y]$ let $I^-=[x,x+e^{-\beta
n}]$ and $I^+=[y-e^{-\beta n},y]$. We have
\[
L_\beta\subset \bigcap_{N>N_0}\bigcup_{n\geq N}\bigcup_{I\in C_n}
(I^-\cup I^+).
\]
The union $\bigcup_{I\in C_n} (I^-\cup I^+)$ has at most
$2e^{2\delta n}$ intervals of length $e^{-\beta n}$ each. Hence,

\[
H^{\frac{2\delta} {\beta}+ \epsilon}_{e^{-\beta n}}
(L_{\beta})\leq 2\sum_{n\geq N}\sum_{I\in C_n}
e^{-n\left(\frac{2\delta\beta}{\beta}\right)}e^{-n\beta\epsilon}\leq
2\sum_{n\geq N}e^{-n\beta\epsilon},
\]
where $H^s_\epsilon$ denotes the approximation of $s$-dimensional
Hausdorff measure using $\epsilon$-covers, and thus $\dim
L_{\beta}\leq \frac{2\delta}{\beta}$. Since $\delta$ was chosen
arbitrarily the proof is complete.
\end{proof}
It immediately follows that $\dim X_{\infty}=\dim\Lambda$.

To complete the proof of part 4 of the theorem we need to show that $\dim X_\alpha$ is
nondecreasing in the presence of parabolic points. Let $\nu_0$ be
a Dirac measure on a parabolic orbit. Then for any invariant
measure $\mu$ measure $\nu_t=t\mu + (1-t) \nu_0$ is also invariant
and satisfies
\[
\frac {h(nu_t,\sigma)} {\lambda(\nu_t)} = \frac {h(\mu,\sigma)}
{\lambda(\mu)}.
\]
Furthermore
\[
-\frac {\int \phi d\nu_t} {\int \psi d\nu_t} = - \frac {\int \phi
d\mu} {\int \psi d\mu} - \frac {1-t} t \frac {\int \phi d\nu_0}
{\int \psi d\mu}
\]
and the right hand side can take any value between $- \frac {\int
\phi d\mu} {\int \psi d\mu}$ and $\infty$. The statement follows.

\section{Proof of Theorem \ref{main3}}
We only need to prove that

\[
\dim X_\alpha \leq
\sup_{\mu\in\mathcal{M}_H(\Sigma_A)}\left\{\frac{h(\mu,\sigma)}{\lambda(\mu)}:-\frac{\int
\phi\text{d}\mu}{\lambda(\mu)}=\alpha\right\}
\]
since the opposite inequality follows from \eqref{ergodic}. For
any $\alpha\in (\alpha_{\min},\infty)$ we choose
$\epsilon<(\alpha-\alpha_{\min})/3$. By Lemma \ref{density},
Theorem \ref{main} and Corollary \ref{sn} we know that for $n$
sufficiently large there exists  measures
$\mu_{q_0,n},\mu_{q_1,n}\in M_n(\Sigma_A)$ with $q_0, q_1$ of
disjoint supports such that

\[
-\frac{\int
S_{n+k}\phi\text{d}\mu_{q_1,n}}{\int S_{n+k}\psi\text{d}\mu_{q_1,n}}\in(\alpha,\alpha+2\epsilon)
\]
\[
\frac{h(\mu_{q_1,n},\sigma^{n+k})}{\int S_{n+k}\psi\text{d}\mu_{q_1,n}}\geq\dim
X_{\alpha+\epsilon}-\epsilon
\]
and
\[
-\frac{\int
S_{n+k}\phi\text{d}\mu_{q_0,n}}{\int S_{n+k}\psi\text{d}\mu_{q_0,n}}\in(\alpha-2\epsilon,\alpha)
\]
\[
\frac{h(\mu_{q_0,n},\sigma^{n+k})}{\int S_{n+k}\psi\text{d}\mu_{q_0,n}}
\geq\dim
X_{\alpha-\epsilon}-\epsilon.
\]

For $0\leq t\leq 1$ we can define $q_t=tq_1+(1-t)q_0$ and consider
the measures $\mu_{q_t,n}\in M_n(\Sigma_A)$. From the properties of
$M_n(\Sigma_A)$ (Lemma \ref{muq} statement 4) it is clear that
$$-\frac{\int
S_{n+k}\phi\text{d}\mu_{q_t,n}}{\int S_{n+k}\psi\text{d}\mu_{q_t,n}}$$
will vary continuously with $t$. Thus there will be a $t_0$ such that
$$-\frac{\int
S_{n+k}\phi\text{d}\mu_{q_{t_0,n}}}{\int S_{n+k}\psi\text{d}\mu_{q_{t_0,n}}}=\alpha.$$

As $q_0$ and $q_1$ have disjoint supports it follows from the second part of Lemma \ref{muq} that
\begin{eqnarray*}
h(\mu_{q_t,n},\sigma^{n+k}) &=& t h(\mu_{q_1,n},\sigma^{n+k}) + (1-t) h(\mu_{q_0,n},\sigma^{n+k}) - t\log t - (1-t)\log (1-t)\\\
  &\geq& t h(\mu_{q_1,n},\sigma^{n+k}) + (1-t)
h(\mu_{q_0,n},\sigma^{n+k}).
\end{eqnarray*}
At the same time,the fifth part of Lemma \ref{muq} implies that
\begin{eqnarray*}
&&\left| \int \frac 1 {n+k} S_{n+k} \psi d\mu_{q_t,n} - t\int \frac 1
{n+k} S_{n+k} \psi d\mu_{q_1,n} -(1-t)\int \frac 1 {n+k} S_{n+k}
\psi d\mu_{q_0,n} \right|\\
&& < \frac {2kL} {n+k} + 2\rho_n.
\end{eqnarray*}
Hence 
\[
\limsup_{n\to\infty} \left( \inf_{t\in
(0,1)}\frac{h(\mu_{q_t,n},\sigma^{n+k})}{\int
S_{n+k}\psi\text{d}\mu_{q_t,n}} - \min \left(
\frac{h(\mu_{q_1},\sigma^{n+k})}{\int
S_{n+k}\psi\text{d}\mu_{q_1,n}} ,
\frac{h(\mu_{q_0,n},\sigma^{n+k})}{\int
S_{n+k}\psi\text{d}\mu_{q_0,n}} \right) \right) \geq 0.
\]
The result now follows by using the continuity of $\dim
X_{\alpha}$ and applying Lemma \ref{approx}, like in Corollary
\ref{ergodicdensity}.

\section{Proof of Theorem \ref{main2}}

Points 1) and 2) of Theorem \ref{main2} are an immediate
consequence of the following lemmas.
\begin{lemma}
Let $\mu_\alpha$ be an equilibrium state for the potential $\psi_a$
and let
\[
\alpha = - \frac {\int \phi d\mu_\alpha} {\lambda(\mu_\alpha)}.
\]
We then have that
\[
\frac {h(\mu_\alpha)} {\lambda(\mu_\alpha)} = \sup_{\mu\in\mathcal{M}_{\sigma}(\Sigma_A)} \left\{\frac
{h(\mu,\sigma)} {\lambda(\mu)}: -\frac {\int \phi d\mu} {\lambda(\mu)}
=\alpha\right\} = b(a) \alpha - a.
\]
\end{lemma}
\begin{proof}
We are going to compare $h/\lambda$ for $\mu_\alpha$ and for any
other invariant measure $\mu$ with $-\int \phi
d\mu/\lambda(\mu)=\alpha$. As $\alpha$ is finite, $\lambda(\mu)$
must be positive. By the variational principle we have that
\[
0=P(\psi_a) =h(\mu_\alpha,\sigma)+\int \psi_a d\mu_\alpha \geq h(\mu,\sigma) +
\int \psi_a d\mu.
\]
Dividing by $\lambda(\mu_\alpha)$ or by $\lambda(\mu)$, we get
\[
0=\frac {h(\mu_\alpha,\sigma)} {\lambda(\mu_\alpha)} + a -b(a)\alpha
\]
and
\[
0\geq \frac {h(\mu,\sigma)} {\lambda(\mu)} + a -b(a)\alpha
\]
from which the assertion follows.
\end{proof}
\begin{lemma}
For any $\alpha\in(\alpha_{\min},\infty)$ there exists some
measure $\mu_\alpha$, which is an equilibrium state for $\psi_a$
for some $a\in \R$, such that
\[
-\frac {\int \phi d\mu_\alpha} {\lambda(\mu_\alpha)} = \alpha.
\]
\end{lemma}
\begin{proof}
As $\int \phi d\mu/\lambda(\mu)$ is a continuous function of $\mu$
(in the weak* topology), we only need to check that there exist
such measures $\mu_\alpha$ for $\alpha$ arbitrarily close to
$\alpha_{\min}$ or $\alpha_{\max}$.

Let us first describe $b(a)$. We will let
$$s=\sup\{t:P(t\psi)=0\}.$$
If $T$ has parabolic points, $P(a\psi)\geq 0$ for any $a\in \R$
and so $b(a)$ is never negative. For $a\leq s$ we have that
$b(a)=0$, otherwise it is strictly positive. On the other hand,
$b(a)$ can be arbitrarily big since
\[
b(a)\geq \frac {P(a\psi)} {-\inf_{\underline{i}}
\phi(\underline{i})}
\]
and so $b(a)\rightarrow\infty$ as $a\rightarrow\infty.$ If $T$
does not have any parabolic points, $b(a)$ ranges from $-\infty$
to $\infty$.

Let $\mu_a$ be an equilibrium state for $\psi_a$ and $\mu$ be any
invariant measure. Denote
\[
\alpha = -\frac {\int \phi d\mu} {\lambda(\mu)}
\]
and
\[
\alpha_a = -\frac {\int \phi d\mu_a} {\lambda(\mu_a)}.
\]
By the variational principle we have
\[
0 = h(\mu_a,\sigma) + \int \psi_a d\mu_a \geq h(\mu,\sigma) + \int \psi_a d\mu.
\]
Dividing by $\lambda(\mu_a)$ or by $\lambda(\mu)$ and subtracting,
we get
\[
0\geq b(a)(\alpha_a-\alpha) + \frac {h(\mu,\sigma)} {\lambda(\mu)} -
\frac {h(\mu_a,\sigma)} {\lambda(\mu_a)}.
\]
The last summand on the right hand side is bounded by 1, the
second summand is positive. Hence,
\[
\alpha_a \leq \alpha +\frac 1 {b(a)}
\]
As $\alpha$ can be chosen to be arbitrarily close to
$\alpha_{\min}$ (by the proper choice of $\mu$) and $b(a)$ can be
chosen to be arbitrarily big (by the proper choice of $a$), this
means that the corresponding $\alpha_a$ can also be chosen
arbitrarily close to $\alpha_{\min}$. Similar argument works for
$\alpha_{\max}$ if $\alpha_{\max}<\infty$.

If $\alpha_{\max}=\infty$, for $a$ small enough $b(a)=0$ and the
Dirac measure at a parabolic orbit is the equilibrium state for
$\psi_a$. Hence, among the equilibrium states for the family
$\{\psi_a\}$ there will be some arbitrarily close (in the weak*
topology) to this Dirac measure. For them $\int \psi$ is
arbitrarily small while $-\int \phi$ is bounded away from zero.
Hence, the ratio $-\int \phi/\int \psi$ is arbitrarily big.
\end{proof}

We proved point 1) of Theorem \ref{main2}, stating that $\dim
X_\alpha$ is a Legendre-Fenchel transformation of the (sometimes multivalued) function
$b^{-1}(a)$. The concavity of $\dim X_\alpha$ follows by standard
properties of the Legendre-Fenchel transformation.

\section{Proof of Theorem \ref{thm:anal}}

In this section we assume the existence of parabolic periodic orbits. To
prove Theorem \ref{thm:anal} we will construct an induced system
(a hyperbolic expanding map with infinitely many inverse branches)
and apply results known for such systems. We construct the induced
countable state system, $\bar{T}:\bar{X}\rightarrow\bar{X}$ for
our parabolic system as in \cite{HMU}. We define the functions
$\bar{\phi},\bar{\psi}$ as the induced potentials relating to
$\psi$ and $\phi$. For $\alpha\in [\alpha_{\min},\infty)$ we
define the set
\[
\bar{X}_{\alpha}=\left\{x\in\bar{X}:
\lim_{n\rightarrow\infty}\frac {\sum_{i=0}^{n-1}\bar{\phi}(T^ix)}
{\sum_{i=0}^{n-1}\bar{\psi}(T^ix)}=\alpha\right\}.
\]

It is immediately clear that
$$X_{\alpha}\subset\bar{X}_{\alpha}\subset \bigcup_{\beta \leq \alpha} \Pi Y_{\beta}.$$
By \eqref{ya} and monotonicity of $\dim X_\alpha$, it immediately
follows that $\dim X_{\alpha}=\dim\bar{X}_{\alpha}$. In the case
where there exists no SRB measure it is shown in Theorem 7.4 of
\cite{HMU} that $\dim \bar{X}_{\alpha}$ varies analytically with
$\alpha$. In the case where the SRB measure $\mu$ exists, Theorem
7.4 of \cite{HMU} shows that $\dim \bar{X}_{\alpha}$ varies
analytically with $\alpha$ when
$\alpha<-\frac{\int\phi\text{d}\mu}{\int\psi\text{d}\mu}.$ For
$\alpha$ above this value the function $\alpha\rightarrow\dim
X_{\alpha}$ is constant and equal to $\dim \Lambda$.


\begin{thebibliography}{99}
\bibitem{BS} L.\ Barreira, B.\ Saussol, {\em Variational principles
and mixed multifractal spectra}, Trans.\ Amer.\ Math.\ Soc.\ {\bf
353} (2001), 3919--3944.
\bibitem{B} W. Byrne,  {\em Multifractal Analysis of Parabolic Rational
Maps}, Phd Thesis, The University of North Texas.
\bibitem{CM}
R. Cawley  and D. Mauldin,
{\em Multifractal decompositions of Moran fractals},
Adv. Math. 92 (1992), no. 2, 196--236.
\bibitem{CLT}G. Contreras, A. O. Lopes, and Ph. Thieullen, {\em Lyapunov minimizing
measures for expanding maps of the circle}, Ergodic Theory Dynam.
Systems {\bf 21} (2001), 1379--1409.
\bibitem{GR1} K.\ Gelfert and M.\ Rams, {\em Geometry of limit set for expansive
Markov systems}, Trans. Amer. Math. Soc. {\bf 361} (2009),
2001-2020.
\bibitem{GR} K.\ Gelfert and M.\ Rams, {\em Multifractal analysis of Lyapunov
exponents of parabolic iterated function systems}, Ergodic Theory
Dynam. Systems {\bf 29} (2009), 919-940.
\bibitem{HMU} P. Hanus, R. Mauldin and  M. Urba\'nski, {\em Thermodynamic
formalism and multifractal analysis of conformal infinite iterated
function systems},  Acta Math. Hungar.  {\bf 96}  (2002), 27--98.
\bibitem{HR} F.\ Hofbauer and P.\ Raith, {\em The Hausdorff dimension of an ergodic
invariant measure for a piecewise monotonic map of the interval},
Canad.\ Math.\ Bull.\ {\bf 35}(1992), no.\ 1, 84--98.
\bibitem{JJOP}
A. Johansson, T. Jordan, A. \"{O}berg, M. Pollicott, {\em
Multifractal analysis of non-uniformly hyperbolic systems}, to appear in The Israel Journal of Mathematics,
preprint available at www.maths.bris.ac.uk/$\sim$matmj/atam103.ps,
2008.
\bibitem{Kess01}
M. Kesseb\"{o}hmer, {\em Large deviation for weak Gibbs measures
and multifractal spectra},  Nonlinearity  {\bf 14}  (2001),
395--409.
\bibitem{KS1}
M. Kesseb\"{o}hmer and B. Stratmann. {\em A multifractal formalism
for growth rates and applications to geometrically finite Kleinian
groups}, Ergodic theory and dynamical systems {\bf 24} (2004),
141--170.
\bibitem{KS2}
M.Kesseb\"{o}hmer and B. Stratmann, {\em A multifractal analysis
for Stern-Brocot intervals, continued fractions and Diophantine
growth rates.}, J. Reine Angew. Math. {\bf 605} (2007), 133--163.
\bibitem{Led}
F. Ledrappier, {\em Some properties of absolutely continuous
invariant measures on an interval}, Ergodic Theory Dynam, Systems
{\bf 1} (1981), 77--93.
\bibitem{N}K. Nakaishi, {\em Multifractal formalism for some parabolic
maps, Ergodic theory and dynamical systems}, {\bf 20} (2000),
843-857.
\bibitem{O} L.\ Olsen, {\em Multifractal analysis of divergence
points of deformed measure theoretical Birkhoff averages}, J.\
Math.\ Pures Appl.\ {\bf 82} (2003), 1591--1649.
\bibitem{P} Y.\ Pesin, {\em Dimension Theory in Dynamical Systems},
Contemporary Views and Applications, Chicago Lectures in
Mathematics. University of Chicago Press, Chicago 1997.
\bibitem{PW} Y.\ Pesin and H.\ Weiss, {\em The multifractal analysis
of Birkhoff averages and large deviations}, Global analysis of
dynamical systems, 419--431, Inst.\ Phys., Bristol, 2001.
\bibitem{R}
D. Rand, {\em The singularity spectrum $f(\alpha)$ for
cookie-cutters}, Ergodic Theory Dynam. Systems 9 (1989), no. 3,
527--541.
\bibitem{SU2} B. Stratmann and M. Urba\'{n}ski, {\em Real Analyticity of Topological Pressure for
parabolically semihyperbolic generalized polynomial-like maps} ,
Indag. Mathem. {\bf 14} (2003), 119-134.
\bibitem{SU} B. Stratmann and M. Urba\'{n}ski, {\em  Multifractal analysis for
parabolically semihyperbolic generalized polynomial-like maps} ,
IP New Studies in Advanced Mathematics, {\bf 5} (2004), 393-347.
\bibitem{U}M. Urba\'nski, {\em Parabolic Cantor sets} Fund. Math. {\bf 151} (1996), 241--277.
\bibitem{W} P.\ Walters, {\em An Introduction to Ergodic Theory }, Springer, 1982.
\bibitem{Yuri00}
M. Yuri, {\em Weak Gibbs measures for certain non-hyperbolic
systems}, Ergodic Theory Dynam. Systems  {\bf 20}  (2000),
1495--1518.
\bibitem{Yuri}M. Yuri, {\em Multifractal analysis of weak Gibbs measures for
intermittent systems}, Comm. Math. Phys. {\bf 230} (2002), no. 2,
365--388.
\end{thebibliography}
\end{document}